\documentclass[11pt]{article} 

\usepackage[latin1]{inputenc} 
\usepackage{geometry} 
\usepackage{array} 

\usepackage{bbm}
\usepackage[all]{xy}
\usepackage{amssymb}
\usepackage{amsmath}
\usepackage{amsthm}
\usepackage{mathrsfs}
\usepackage{enumitem}

\theoremstyle{plain} 
\newtheorem{thm}{Theorem}[section]

\newtheorem{prop}[thm]{Proposition}
\newtheorem{lem}[thm]{Lemma}

\theoremstyle{definition} 
\newtheorem{definition}[thm]{Definition}
\newtheorem{remarque}[thm]{Remark}
\newtheorem{notation}[thm]{Notation}

\newcommand{\actionG}{\underline{\delta}}
\newcommand{\brep}{\alpha}
\newcommand{\C}{\mathbbm{C}}
\newcommand{\CAlg}{\mathcal{O}}
\newcommand{\CatInd}[1][\brep]{\mathcal{T}_{#1}}

\newcommand{\ChainGp}{\mathfrak{C}}

\newcommand{\CQG}{\mathbbm{G}}
\newcommand{\E}{\mathbbm{E}}

\newcommand{\Falg}[1][\alpha]{{\mathcal{O}^{#1}}}
\newcommand{\Hh}{\mathscr{H}}
\newcommand{\HilbK}{\mathscr{K}}

\newcommand{\inj}{\hookrightarrow}
\newcommand{\InvJ}{\mathcal{F}}
\newcommand{\LieAlgSU}{\mathfrak{su}}

\newcommand{\N}{\mathbbm{N}}
\newcommand{\OpCpct}{\mathcal{K}}
\newcommand{\OpLin}{\mathcal{L}}

\newcommand{\R}{\mathbbm{R}}
\newcommand{\Toep}{\mathcal{T}}
\newcommand{\TrivRep}{\epsilon}
\newcommand{\UCTclass}{\mathscr{N}}

\newcommand{\Z}{\mathbbm{Z}}

\DeclareMathOperator{\coker}{Coker}

\DeclareMathOperator{\id}{Id}
\DeclareMathOperator{\Span}{Span}

\renewcommand{\colon}{\mathclose{}\mathpunct{}\mathpunct{:}}

\numberwithin{equation}{section}

\usepackage{hyperref}

\title{Fixed~points of compact~quantum~groups actions on~Cuntz~algebras}
\author{Olivier \textsc{Gabriel}}

\begin{document}
\maketitle

\abstract{
Given an action of a Compact Quantum Group (CQG) on a finite dimensional Hilbert space, we can construct an action on the associated Cuntz algebra. We study the fixed point algebra of this action, using Kirchberg classification results and Pimsner algebras. Under certain conditions, we prove that the fixed point algebra is purely infinite and nuclear. We further identify it as a Pimsner algebra, compute its $K$-theory and prove a ``stability property'': the fixed points only depend on the CQG \textsl{via} its fusion rules. We apply the theory to $SU_q(N)$ and illustrate by explicit computations for $SU_q(2)$ and $SU_q(3)$. This construction provides examples of free actions of CQG (or ``principal noncommutative bundles'').
}

\medbreak



\tableofcontents

\section{Introduction}

Our original motivation for this article was to find an explicit construction of free actions of Compact Quantum Groups (CQG -- see Woronowicz' articles \cite{TwistedSU2Woronowicz,CpctMPseudoGpWoronowicz,CpctQGpWoronowicz}) in the sense of Ellwood's work \cite{PpalActionEllwood} (see also the recent article \cite{FreeActCAlgCY} by De Commer and Yamashita). We rely on a construction originally given by Konishi, Nagisa and Watatani \cite{ActCMQGKNW}. Starting with a unitary representation $\brep $ of dimension $d$ of a CQG $\CQG $, we obtain an action of $\CQG $ on the Cuntz algebra \cite{CuntzAlg} $\CAlg_d$.

Another motivation for the study of such fixed point algebras is provided by the duality results of Doplicher and Roberts \cite{CuntzAlgDR,EndomDR} for actions of (ordinary) compact groups on $C^*$-algebras and the considerable amount of work they generated (see for instance \cite{HilbertCSystBaumgLledo,DualThmPinzariRoberts,BraidedCatPinzariRoberts}).

The previous construction was studied by Marciniak in \cite{CpctQGpMarciniak}, in the case of the irreducible representation $\brep $ of $\CQG = SU_q(2)$ on $\C^2$. He gave an explicit description of the fixed point algebra by means of generators and relations. The general case of the natural irreducible representation of $\CQG = SU_q(d)$ on $\C^d$ was treated by Paolucci in her article \cite{CoactionCuntzAlgPaolucci}. She gave an explicit family of generators based on the infinite braid group $B_\infty $. She further expanded her results in joint work \cite{QGpActCPZ} with Carey and Zhang. Working in the algebraic setting, they gave the same kind of description in the case of classical quantum groups $\CQG = SU_q(d)$, $SO_q(d)$ or $SP_q(d)$. Moreover, they managed to recover the ``$q$'' in $(0,1)$ of $\CQG $ from the fixed point algebra.

We take a rather different approach to this problem. Indeed, we stay strictly in the $C^*$-algebraic domain and try to obtain an abstract $C^*$-isomorphism. For this abstract identification, we depend on the remarkable classification results by Kirchberg and Phillips (see \cite{ClassThmKirchberg,ClassThmKirchbergKP}) allowing recovery a $C^*$-isomorphism from $K$-theoretic properties. In this article, our reference regarding classification theory is the book \cite{RordamStormer} by R\o{}rdam and St\o{}rmer. 

To perform the required $K$-theoretic computations, we use the notion of Pimsner algebra (also called Cuntz-Pimsner algebra) which first appeared in \cite{PiCrG}. A very similar notion was introduced independently by Abadie, Eilers and Exel in \cite{AbadieEE}. In our setting both definitions coincide. If the version by Pimsner is more widely known, our approach is somewhat closer to the second definition. Our exposition is also influenced by the work \cite{Katsura04} of Katsura -- on which we rely for nuclearity.

\smallbreak

Our article includes a proof that a certain Pimsner algebra is purely infinite. It is therefore akin to \cite{SimpleCAlgPinzari,SimpAlgCP,QfreeAutZacharias,KumjianPiCrG} as we now discuss. 

If the results of Pinzari in Theorem 3.12 of \cite{SimpleCAlgPinzari} are very similar to ours, her setting of regular representations of (ordinary) compact groups is disjoint from ours. In a subsequent joint work with Kajiwara and Watatani, she provided a sufficient condition for purely infiniteness (see \cite{SimpAlgCP} Theorem 6.1). However, this criterion only applies to the case of purely infinite coefficient algebras and our coefficients are Approximately Finite dimensional (AF-) algebras (see our Lemma \ref{Lem:AlgApprox}), thus their theory does not apply to our case.

Zacharias provides another criterion for purely infiniteness before applying it to crossed products by $\R$. We were not able to apply Proposition 2.2 of \cite{QfreeAutZacharias} to our situation, despite similarities between both arguments. For further discussion, see our Remark \ref{Rk:Zacharias}. 

Finally, Kumjian also considered purely infinite Pimsner algebras in Theorem 2.8 of \cite{KumjianPiCrG}. In his setting, the natural embedding of the coefficient algebra inside the Cuntz algebra induces a $KK$-equivalence, thus the $K$-theories of both algebras should be isomorphic. The explicit example of the natural representation of $\CQG =SU_q(2)$ (see our Subsection \ref{SSec:Example}) shows that this need not be the case in our setting.

\smallbreak

In the course of this article, we resort to previous work by Banica \cite{FusionRulesBanica} to define $R^+$-isomorphic CQG and by Wassermann \cite{TheseWassermann} for the computation of $K$-theory.

\medbreak

The main results of this article come from two directions: 
\begin{itemize}
\item
Under some hypotheses (\textsl{e.g.} $\CQG $ semisimple compact Lie group and $\brep $ sum of two irreducible representations), we can prove that the fixed point algebra $\Falg $ is simple and purely infinite (Theorem \ref{Thm:PI1});
\item
If we can find $N$ large enough for the tensor representation $\brep ^N$ to be included in $\brep ^{N+1}$, then $\Falg $ is a Pimsner algebra (Proposition \ref{Prop:PimsnerAlg}) and we can compute its $K$-theory (Theorem \ref{Thm:KthFalg}).
\end{itemize}
Combining these two threads of results, we can prove an unexpected ``stability theorem'' (Theorem \ref{Thm:Main}) namely that as a $C^*$-algebra, the fixed point algebra $\Falg $ only depends on $\CQG $ \textsl{via} its fusion rules. We show that these results apply to the natural representation of $SU_q(N)$ and discuss two explicit examples, the cases of $\CQG = SU_q(2)$ and $\CQG = SU_q(3)$. In this setting, this ``stability result'' for $C^*$-algebras contrasts sharply with the algebraic case treated in \cite{QGpActCPZ}, as we discuss in Remark \ref{Rem:CPZ}. Finally, we prove that this construction indeed yields free actions of CQG.

\medbreak

After reviewing the construction of the fixed point algebra $\Falg $ and stating properly the hypotheses of our theorems (Section \ref{Sec:Falg}), we prove the simplicity and purely infiniteness properties in Section \ref{Sec:SimplePI}. We provide a brief reminder regarding Pimsner algebras, before proceeding with the computation of $K$-theory (Section \ref{Sec:CompKth}). The next section \ref{Sec:Main} focuses on the combination of both parts to obtain the main results. We then discuss our hypotheses, show how to apply our theory to $SU(N)$ and study two explicit examples (Section \ref{Sec:Discussion}). Going back to our original motivation, we conclude in Section \ref{Sec:Final} with sufficient conditions to get a free action and a comparison with the algebraic case.

\section{Fixed points: review and conditions}
\label{Sec:Falg}

In this article, all tensor products of $C^*$-algebras are \emph{minimal} tensor products. We consider a Compact Quantum Group (CQG) denoted $\CQG$, \textsl{i.e.} a separable unital $C^*$-algebra $C(\CQG)$ together with a unital $*$-algebra homomorphism $\Delta : C(\CQG) \to C(\CQG) \otimes C(\CQG)$ which satisfies coassociativity and cancellation properties -- for more details on these objects and their representations, see \cite{CpctMPseudoGpWoronowicz,CpctQGpWoronowicz}. 

We also fix a \emph{unitary representation} $\brep$ of $\CQG $ on a Hilbert space $\Hh$ of finite dimension $d$ \textsl{i.e.} a $C(\CQG )$-valued $d \times d$ matrix $(\brep_{ij}) \in M_d(C(\CQG ))$ satisfying a coassociativity property as well as equations $\sum_{k} \brep _{k i}^* \brep _{kj} = \delta_{ij}$ and $\sum_{k} \brep _{i k} \brep _{j k}^* = \delta_{ij}$. In particular, for any $\CQG $, we have a trivial representation denoted $\TrivRep$, defined by the unit of $C(\CQG )$. In the rest of this article, we only use unitary representations which we therefore simply call ``representations''. 

If $(e_i)$ is an orthonormal basis of $\Hh$, we define the associated \emph{action} of $\CQG $ on $\Hh$ as the map $\delta_\alpha \colon \Hh \to \Hh \otimes C(\CQG )$ given by $\delta_\alpha(e_i) = \sum_{j} e_j \otimes \alpha_{ji}$. Note that our conventions differ from those of \cite{CoactionCuntzAlgPaolucci,FusionRulesBanica}.

Using this definition of representations for CQG, Woronowicz introduced several notions which are very similar to the case of compact groups. In particular, given two representations $\alpha \in M_{d_1}(C(\CQG ))$ and $\beta \in M_{d_2}(C(\CQG ))$, a linear map $T $ between $\Hh_1$ and $\Hh_2$ is an \emph{intertwiner} (see Proposition 2.1 p.629 of \cite{CpctMPseudoGpWoronowicz}) if
$$ (T \otimes 1) \delta_\alpha = \delta_\beta T. $$
He also defined \emph{irreducible representations} and proved a \emph{Schur's lemma}. Given two unitary representations $\pi_1$ and $\pi_2$ of $\CQG $ on Hilbert spaces $\Hh_1$ and $\Hh_2$ respectively, he moreover constructed a unitary tensor product representation $\pi_1 \otimes \pi_2$ on $\Hh_1 \otimes \Hh_2$ -- see Section 2 of \cite{CpctMPseudoGpWoronowicz}. Furthermore, Woronowicz defined and proved the existence of \emph{Haar measures} for general CQG -- see \cite{CpctQGpWoronowicz} Theorem 1.3.

We denote by $R^+(\CQG )$ the \emph{fusion semiring} of (equivalence classes of) finite dimensional representations of $\CQG $ equipped with addition and tensor product (compare Def. 1.2, \cite{FusionRulesBanica}). The following is an avatar of Definition 2.1 in \cite{FusionRulesBanica}:
\begin{definition}
\label{Def:IsomRepCat}
Given two CQG $\CQG_1$ and $\CQG_2$, we say that they are \emph{$R^+$-isomorphic} if an isomorphism of semirings $ R^+(\CQG _1) \simeq R^+(\CQG_2)$ exists.
\end{definition}
In other words, $\CQG _1$ and $\CQG _2$ have the same \emph{fusion rules}, \textsl{i.e.} there is a bijection $\Phi$ between their irreducible representations and it is compatible with the tensor product. As an example, the $q$-deformations $SU_q(N)$ are $R^+$-isomorphic to the original $SU(N)$ (see \cite{FusionRulesBanica}, Theorem 2.1 and references therein).

\medbreak

The Cuntz algebra $\CAlg_d$ is defined \cite{CuntzAlg} as the universal unital $C^*$-algebra generated by $d$ elements $(S_j)_{1 \leqslant j \leqslant d}$ which satisfy $S_i^* S_j = \delta_{ij}$ and $\sum_{j} S_j S_j^* = 1$. If $(e_i)$ is an orthonormal basis of a Hilbert space $\Hh$ of dimension $d$, we define an injective linear map $\varphi \colon \Hh \to \CAlg_d$ by $e_i \mapsto S_i$. Indeed, this map preserves the scalar product in the sense that if $v, w \in \Hh$, then 
$$ \langle v, w \rangle = \varphi(v)^* \varphi(w) \in \C 1  \subseteq \CAlg_d .$$
Following \cite{CpctQGpMarciniak}, we extend this map to iterated tensor products and set
$$ \Hh^\ell (\Hh^*)^k := \Span \{ S_{i_1} \cdots S_{i_\ell} S^*_{j_1} \cdots S_{j_k}^*| 1 \leqslant i_1, \ldots , i_\ell, j_1, \ldots , j_k \leqslant d \} \subseteq \CAlg_d .$$
Taking $(\ell,k) = (1,0)$ (resp. $(\ell, k) = (0,1)$) we recover $\Hh$ (resp. $\Hh^*$). These spaces satisfy $\Hh^* \Hh \subseteq \C 1 \subseteq \CAlg_d$. Thus, we can identify $\Hh^\ell (\Hh^*)^k$ with the linear operators from $\Hh^k$ to $\Hh^{\ell}$. Moreover, an inclusion $\Hh^\ell (\Hh^*)^k \inj \Hh^{\ell+1} (\Hh^*)^{k+1}$ is defined by 
$$
S_{i_1} \cdots S_{i_\ell} S^*_{j_1} \cdots S_{j_k}^* \mapsto \sum_{p} S_{i_1} \cdots S_{i_\ell} S_p S_p^* S^*_{j_1} \cdots S_{j_k}^*.
$$
We make extensive use of these identifications. We denote by $\CAlg_d^{\text{alg}}$ the algebraic unital $*$-algebra generated by $S_j$. $T \in \CAlg_d^{\text{alg}}$ is called an \emph{algebraic element}.

\smallbreak

The Cuntz algebra also possesses a well-known \emph{canonical endomorphism} $\rho \colon \CAlg_d \to \CAlg_d$ defined by:
$$ \rho(T) = \sum_{j=1}^d S_j T S_j^* ,$$
such that for any $S \in \Hh$ and any $T \in \CAlg_d$,
$ 
S  T = \rho(T) S
$.
Consequently, given $T \in \CAlg_d$ and $T' \in \CAlg_d \cap  \Hh^{\ell} (\Hh^*)^k$, setting $\mu = \max \{ \ell, k\}$ and $n = \ell - k$, we get that for any $\Lambda \geqslant \mu$,
\begin{equation}
\label{Eqn:CommRho}
 T' \rho^\Lambda(T) = \rho^{\Lambda +n}(T) T' .
\end{equation}
A pointwise continuous $S^1$-action $\gamma$ on a $C^*$-algebra $B$ is called a \emph{gauge action}. It yields \emph{spectral subspaces} $B^{(k)}$ defined by:
$$ B^{(k)} := \{T \in B | \forall z \in S^1 \subseteq \C, \gamma_z(T) = z^k T \} .$$ 
$T \in B^{(k)}$ is called a \emph{gauge-homogeneous} element of \emph{total gauge} $k$.

In the case of $\CAlg_d$, such a gauge action is defined on generators by $\gamma_z(S_j) = z S_j$, yielding spaces $\CAlg_d^{(k)}$. We use the short hand notation $\InvJ := \CAlg_d^{(0)}$ and call the elements of this set \emph{gauge-invariant}. This algebra contains $\InvJ^\ell := \Hh^\ell (\Hh^*)^\ell$, which can be identified with the matrices $M_{d^\ell}(\C)$ by the previous discussion.

\medbreak

An action $\actionG$ of a CQG on a $C^*$-algebra $B$ is required to be compatible with the product and the involution $*$ of $B$:
\begin{align*}
\actionG(T_1 T_2) &= \actionG(T_1) \actionG(T_2)
&
\actionG(T)^* &= \actionG(T^*).
\end{align*}
(compare Def. 1, \cite{CoactionCuntzAlgPaolucci}). The following first appeared as Theorem 1 in \cite{ActCMQGKNW}:
\begin{thm}
For a given representation $\brep $ of $\CQG $ on a Hilbert space $\Hh$ of dimension $d$, there is an action $\actionG \colon \CAlg_d \to \CAlg_d \otimes C(\CQG )$ of $\CQG $ on $\CAlg_d$ defined by
$$ \actionG(S_i) = \sum_{j=1}^d S_j \otimes \brep _{ji} .$$
\end{thm}
By construction, $\actionG$ preserves the spaces $\Hh^\ell (\Hh^*)^k$. It also induces actions $\delta_k$ of $\CQG $ on $\Hh^k$. The associated representations are called $\brep ^k$ -- indeed, they are the iterated tensor products of $\brep $.

\smallbreak

We now introduce our main object of interest:
\begin{definition}
The fixed point algebra $\Falg$ for the action $\actionG$ is defined as:
$$ \Falg := \{ T \in \CAlg_d : \actionG(T) = T \otimes 1 \} \subseteq \CAlg_d .$$
\end{definition}
In particular, if $\alpha$ is the trivial representation on $\C^d$, given by $\alpha_{ij} = \delta_{ij}$, then we recover $\Falg = \CAlg_d$.

\smallbreak

It is readily checked \cite{CoactionCuntzAlgPaolucci} that the gauge action $\gamma$ on $\CAlg_d$ ``commutes'' with the coaction of $\CQG $ on $\CAlg_d$, therefore $\Falg $ also carries a gauge action and we set $\InvJ^\brep := \Falg  \cap \InvJ$. We use the notation $\InvJ^{\brep, \ell} := \Falg \cap \Hh^{\ell} (\Hh^*)^{\ell}$. The following already appeared as Proposition 3.4 of \cite{CpctQGpMarciniak}: 
\begin{lem}
\label{Lem:InterpEntrelac}
The elements of $\Falg \cap \Hh^\ell (\Hh^*)^k$ are precisely the intertwiners of the representations $\brep^\ell$ and $\brep^k$.
\end{lem}

We adapt Lemma 5 of \cite{ActCMQGKNW}:
\begin{definition}
If $\CQG $ is a compact quantum group, the \emph{conditional expectation $\E_\CQG $} associated to the action $\brep $ is:
$$ \E_\CQG (T) = (\id_{\CAlg_d} \otimes h)\brep (T) ,$$
where $h$ is the Haar measure on $\CQG $.
\end{definition}
$\E_\CQG $ defines a projection of norm $1$ from $\CAlg_d$ to $\Falg $ (see \cite{CpctQGpMarciniak} p.611). It sends algebraic elements to algebraic elements and $\Hh^{\ell} (\Hh^*)^{\ell}$ to $\Hh^{\ell} (\Hh^*)^{\ell} \cap \Falg $.

We refer to the book \cite{RordamStormer} (Def. 1.2.1) 
for a definition of Approximately Finite dimensional algebras or AF-algebras. The conditional expectation $ \E_\CQG $ is used in the proof of the following Lemma 6 of \cite{ActCMQGKNW}:
\begin{lem}
\label{Lem:AlgApprox}
If $\brep $ is a representation of the CQG $\CQG $, then the algebraic elements of $\Falg $ are dense in $\Falg $. Moreover,
\begin{itemize}
\item
any positive $T \in \Falg $ can be approximated by a positive and algebraic $T_\brep  \in \Falg $;
\item
 $\InvJ^\brep$ is an AF-algebra and $\InvJ^\brep =\lim_{\to } \InvJ^{\brep , \ell}$.
\end{itemize}
\end{lem}

\begin{proof}
Take an element $T \in \Falg  \subseteq \CAlg_d$. By definition of $\CAlg_d$, we can find an algebraic element $T' \in \CAlg_d$ s.t. $\| T' - T \| \leqslant \varepsilon$. Using the conditional expectation $\E_\CQG $ associated to $\brep $ and setting $T_\brep  := \E_\CQG (T')$ the following holds: 
\begin{equation}
\label{Eqn:EA}
\| T_\brep - T \| = \| \E_\CQG(T') - T \| = \| \E_\CQG(T' - T) \| \leqslant \| T' - T \|  \leqslant \varepsilon.
\end{equation}
Hence $T_\brep  \in \Falg $ is an algebraic approximation of $T$ in $\Falg $.

\smallbreak

If $T \in \Falg $ is positive, we can consider its square root $B \in \Falg $ with $T = B^* B$ and $B = B^*$. Applying the above argument to $B$, we get an algebraic approximation $B_\brep $ of $B$. Finally, $T = B^* B$ is approximated by the positive algebraic element $T_\brep := B_\brep ^* B_\brep  \in \Falg $.

\smallbreak

Since $\InvJ$ is an AF-algebra and $\InvJ = \lim_{\to } \InvJ^\ell$, for any $T \in \InvJ^\brep \subseteq \InvJ$ and any $\varepsilon>0$, there is a $T' \in \InvJ^\ell$ for $\ell$ large enough such that $\| T - T' \| \leqslant \varepsilon$. Considering then $T_\brep := \E_\CQG(T')$ and using the estimate \eqref{Eqn:EA}, we see that $T_\brep $ is an algebraic approximation of $T$. Moreover, since $T' \in \InvJ^\ell$, $T_\brep \in \InvJ^{\brep , \ell}$ and the property is proven.
\end{proof}

\begin{notation}
\label{Not:TBrep}
Given a representation $\brep $ of $\CQG $, we call $\CatInd$ the set of (classes of) irreducible representations appearing in the iterated tensor products $\brep^{\ell}$ for $\ell \in \N$.
\end{notation}

Our results rely on the following conditions:
\begin{enumerate}[label=(C\arabic*)]
\item
\label{Cond:Contra}
For any $ \beta \in \CatInd $, we can find $\beta' \in \CatInd$ s.t. the representation $\beta \otimes \beta'$ possesses a nonzero invariant vector.
\item
\label{Cond:Dim}
The representation $\brep $ is not a \emph{single} irreducible representation.
\item
\label{Cond:Entrelac}
There is an integer $N \in \N$ such that $\brep^N$ is contained in $\brep^{N+1}$. 
\end{enumerate}

A few comments on these conditions:
\begin{itemize}
\item
Condition \ref{Cond:Contra} is discussed in Section \ref{Sec:Discussion} -- see Propositions \ref{Prop:SSLieGp} and \ref{Prop:FiniteGp}.
\item 
Condition \ref{Cond:Dim} is satisfied if $\brep = 2 (t)$ where $(t)$ is any irreducible representation.
\item
Given an inclusion of representations $\brep \leqslant  \beta$, it is clear that for any representation $\gamma$, $\brep \otimes \gamma \leqslant \beta \otimes \gamma$. Hence, if Condition \ref{Cond:Entrelac} is satisfied for $N_0$, then it is satisfied for any larger $N \geqslant N_0$.
\end{itemize}

\section{Simple and purely infinite algebras}
\label{Sec:SimplePI}

In this section, we prove that under Conditions \ref{Cond:Contra} and \ref{Cond:Dim}, the algebras $\Falg$ are simple and purely infinite. The gauge action $\gamma$ plays an instrumental part in our argument.
\begin{definition}
A $C^*$-algebra $C$ is called \emph{simple and Purely Infinite} (PI) if for any nonzero $T \in C$ and any $\varepsilon >0$, $\exists A, B \in C$ s.t. $\| A T B - 1 \| \leqslant \varepsilon$.
\end{definition}
It is clear from the above definition that $C$ must be simple. 

\begin{prop}
\label{Prop:Proj}
If $\brep $ is a representation of $\CQG $ which satisfies \ref{Cond:Contra}, then for any nonzero gauge-invariant projection $P \in \Falg \cap \Hh^\ell (\Hh^*)^{\ell}$, we can find $L$ large enough and $A \in \Falg \cap \Hh^L$ s.t. $P A = A$, $ A^* P A = 1 $ and $ A^* A = 1$.
\end{prop}

\begin{proof}
A projection $P \in \Falg$ in $\Hh^{\ell} (\Hh^*)^{\ell}$ always defines a $\brep $-invariant Hilbert space in $\Hh^{\ell}$ -- and therefore a representation of $\CQG $. Indeed, the finite dimensional Hilbert space $\HilbK := P \Hh^{\ell}$ satisfies:
\begin{equation}
\label{Eqn:StableK}
 \actionG(P x) = \actionG(P) \actionG(x) =  (P \otimes 1) \actionG(x) \in \HilbK \otimes C(\CQG ) ,
\end{equation}
thereby inducing a representation of $\CQG $ by restriction. By construction, the restriction of $\actionG $ to $\HilbK$ has a decomposition into irreducible representations which only involves (subrepresentations of) iterated tensor products of $\brep $, \textsl{i.e.} elements of $\CatInd $. Pick one such irreducible representation $\beta \in \CatInd $. From Condition \ref{Cond:Contra}, we get a $q \in \N$ and a $\beta'$ in the decomposition of $\brep ^q$ such that $\beta \otimes \beta'$ possesses an invariant vector.

\smallbreak

Using an argument similar to \eqref{Eqn:StableK}, it is easy to show that $\HilbK \otimes \Hh^{q} \simeq P \Hh^{\ell +q}$ carries a representation of $\CQG $ induced from $\brep $. If we decompose this representation, using the tensor product $\HilbK \otimes \Hh^q$ we find:
$$ ( \beta \oplus  t ) \otimes ( \beta' \oplus u ) \simeq (\beta \otimes \beta') \oplus (\beta \otimes u) \oplus (t \otimes \beta') \oplus (t \otimes u) $$
where $t$ and $u$ are sums of irreducible representations. Property \ref{Cond:Contra} thus implies that $\HilbK \otimes \Hh^q$ contains a nonzero invariant vector.

\smallbreak

Let us pick such an invariant vector $A \neq 0$ in $\HilbK \otimes \Hh^q \simeq  P \Hh^{\ell + q} \subseteq \Hh^{\ell + q}$. Without loss of generality, we can assume that $\| A \| = 1$, \textsl{i.e.} $A^* A = 1$. By construction, $\actionG(A) = A \otimes 1$ so $A \in \Falg $. Since $A$ is in $P \Hh^{\ell +q}$, we have $P A = A$. Hence
\begin{align*}
A^* A &= 1
&
A^* P A &= (PA)^* P A = A^* A = 1.
\end{align*}
\end{proof}

The following proposition implies that $E = (\Falg )^{(1)}$ over $A = \InvJ^\brep $ is aperiodic in the sense of Definition 2.1 \cite{QfreeAutZacharias}.
\begin{prop}
\label{Prop:Aperiod}
If $\brep $ satisfies Condition \ref{Cond:Contra}, then for any positive nonzero $T_0 \in \InvJ^{\brep,\ell}$, we can find an integer $L$ large enough and $A \in \Hh^L$ s.t.
$$ A^* T_0 A = 1. $$
Moreover, we can choose $A$ with $\| A \| = \| T_0 \|^{-1/2}$.
\end{prop}

\begin{proof}
We fix a positive $T_0 \in \InvJ^{\brep,\ell}$ with $T_0 \neq 0$. Since $T_0$ is a normal element of the finite dimensional algebra $ \InvJ^{\brep,\ell}$, we can write it as a finite sum:
$$ T_0 = \sum_{i} \lambda_i P_i ,$$
where the $P_i$ are its spectral projections and the $\lambda_i$ are its eigenvalues. Pick a nonzero eigenvalue $\lambda_j$ and apply Proposition \ref{Prop:Proj} to $P_j$ to get $A \in \Hh^L \cap \Falg $ s.t. $A^* P_j A = 1$. Since $P_j$ is orthogonal to the other spectral projections, for $i \neq j$, we have:
$$ A^* P_i = (P_j A)^* P_i = A^* P_j P_i = 0$$ 
and thus:
$$ A^* T_0 A = A^* \left(\sum_{i} \lambda_i P_i \right) A = \lambda_j A^* P_j A = \lambda_i 1.$$
We just have to renormalise $A$ to get $1$. To estimate $\| A' \|$ for the renormalised $A'$, note that $\InvJ^\ell$ is finite dimensional and $T_0$ positive, thus if we pick the largest eigenvalue $\lambda_i = \| T_0 \|$ we get $\| A' \| = \| T_0 \|^{-1/2}$.
\end{proof}

We follow the line of argument of the proof of Theorem 4.2.2 of \cite{RordamStormer} (compare p.73 therein) and use ``cutting projections'', which we now construct:
\begin{lem}
\label{Lem:QN}
If $\alpha$ satisfies \ref{Cond:Dim}, then for any $N \in \N \setminus \{ 0 \}$, there is a $Q_N \in \InvJ^\alpha$ such that
\begin{itemize}
\item
$Q_N$ is a nonzero projection;
\item 
 $Q_N \, \rho^N(e) = 0$;
\item
for any $0 < n \leqslant N$, we have $ Q_N \rho^{n}(Q_N) = 0$.
\end{itemize}
\end{lem}

\begin{proof}
Since $\brep $ satisfies Condition \ref{Cond:Dim}, we can find a projection $e \in \Hh^1 (\Hh^*)^1$ s.t. $e \neq 1$ and $1 - e \neq 1$: it suffices to consider the projection onto a single irreducible representation in $\brep $. In particular, for any $T \in \CAlg_d$, \eqref{Eqn:CommRho} shows that $\rho(T) e = e\rho(T)$.

We construct $Q_N$ by induction. For $N = 1$, we take $Q_1 := e \rho(1 - e)$. This is clearly a nonzero projection. We can evaluate $Q_1 \rho(e)= e \rho((1-e) e) = 0$ and thus $ Q_1 \rho(Q_1) = Q_1 \rho(e) \rho^2(1-e) = 0 $. 

Now if we have $Q_N$, we construct $Q_{N+1}$ by setting $Q_{N+1} := e \rho(Q_N)$. Since $Q_N \neq 0$, we have $Q_{N+1} \neq 0$. It is not difficult to prove that $Q_{N+1}$ is a projection. Moreover,
$$ Q_{N+1} \rho^{N+1}(e) = e \rho(Q_N \rho^N(e)) =0. $$
Now, for $0 < n \leqslant N$, we have:
$$  Q_{N+1} \rho^{n}(Q_{N+1}) = e \rho(Q_N) \rho^n( \rho(Q_N) e) = e \rho\Big( Q_N \rho^n(Q_N) \rho^{n-1}(e) \Big) = 0 .$$
For $N+1$:
$$  Q_{N+1} \rho^{N+1}(Q_{N+1}) = e \rho(Q_N) \rho^{N+1}(e \rho(Q_N)) = e \rho( Q_N \rho^N(e \rho(Q_N))) =0.$$
We thus get a family of cutting projections.
\end{proof}

\begin{definition}
The \emph{conditional expectation} $\E_{S^1}$ associated to the gauge action is defined by:
$$ \E_{S^1}(T) = \int_{S^1} \gamma_z(T) d \lambda(z) .$$
\end{definition}

\begin{prop}
\label{Prop:Cut}
If $\alpha$ satisfies \ref{Cond:Dim} and if $T \in \Falg$ is an algebraic element, there is a projection $Q \in \InvJ^\alpha$ such that 
\begin{align*}
Q T Q &= Q \E_{S^1}(T) Q 
&
\| Q \E_{S^1}(T) Q \|&=\| \E_{S^1}(T) \| \leqslant \| T \|.
\end{align*}
\end{prop}

\begin{remarque}
It is well known that $\E_{S^1}$ is a faithful conditional expectation (Remark 2 in 1.10 of \cite{CuntzAlg}). Since the gauge action of $S^1$ ``commutes'' with $\actionG$, $\E_{S^1}$ sends $\Falg$ to $\Falg$ (see Proposition 4.3 of \cite{CpctQGpMarciniak}). In particular, if $T$ is positive and nonzero, then $Q \E_{S^1}(T) Q$ is also positive and nonzero.
\end{remarque}

\begin{proof}
We start with an argument similar to \cite{RordamStormer} p.73. $T$ can be written as a (finite) linear combination of reduced words $W_\beta = S_{i_1} \cdots S_{i_{k_\beta}} S^*_{j_1} \cdots S_{j_{m_\beta}}^* \in \Hh^{k_\beta} (\Hh^*)^{m_\beta}$. We then define 
\begin{align*}
\Lambda_0 &:= \max_\beta \min \{ k_\beta, m_\beta \}
&
N &:= \max_\beta \{ |k_\beta - m_\beta | \}.
\end{align*}
In other words, $N$ is the absolute value of the highest total gauge which appears in $T$. We can then decompose $T$ into a finite sum of gauge-homogeneous elements:
$$ T = \sum_{n > 0} T_n + T_0 + \sum_{n<0} T_n ,$$
where each $T_n \in (\Falg )^{(n)}$ and only a finite number of $T_n$ are nonzero. In particular, if $T_n \neq 0$, then $|n| \leqslant N$. We compute $Q T Q$ with $Q = \rho^\Lambda(Q_N)$, where $\Lambda = \Lambda_0 + N$, using \eqref{Eqn:CommRho} and Lemma \ref{Lem:QN}:
\begin{multline*}
Q T Q = \rho^\Lambda(Q_N) T \rho^\Lambda(Q_N) \\
=\sum_{n >0} \rho^\Lambda(Q_N) T_n \rho^\Lambda(Q_N) + Q T_0 Q + \sum_{n <0} \rho^\Lambda(Q_N) T_n \rho^\Lambda(Q_N) \\
=\sum_{n >0}\rho^{\Lambda}(Q_N) \rho^{\Lambda + n}(Q_N) T_n + Q T_0 Q + \sum_{n <0} T_n \rho^{\Lambda+n}(Q_N) \rho^\Lambda(Q_N)\\
=\sum_{n >0}\rho^{\Lambda} \big(Q_N \rho^{n}(Q_N)\big) T_n + Q T_0 Q + \sum_{n <0}  T_n \rho^\Lambda \big( \rho^{n}(Q_N) Q_N \big) \\
= Q T_0 Q = Q \E_{S^1}(T) Q.
\end{multline*}
The same argument as in the proof of Proposition 1.7 of \cite{CuntzAlg} (see p.177) shows that $T \mapsto Q T Q$ induces an isomorphism between $\InvJ^\ell$ and $Q \InvJ^\ell Q$ for $\ell = \Lambda$. In particular, $\| T_0 \| = \| Q T_0 Q \|$.
\end{proof}

\begin{thm}
\label{Thm:PI1}
Let $\brep $ be a representation of $\CQG $ which satisfies Conditions \ref{Cond:Contra} and \ref{Cond:Dim}, the fixed point algebra $\Falg $ is simple and PI.
\end{thm}

\begin{proof}
It suffices to find $A'$ s.t. $ \| (A')^* T^* T A' - 1 \| \leqslant \varepsilon$, since then we can take $(A')^* T^* = A$ and $B = A'$. Hence we only treat the case of $T$ strictly positive. Without loss of generality, we can further assume that $\| \E_{S^1}(T) \| = 1$.

\smallbreak

Using Lemma \ref{Lem:AlgApprox} on $T$, we get an algebraic, positive $T_\brep \in \Falg$ s.t
$$ \| T - T_\brep \| \leqslant \varepsilon .$$
In particular, for $\varepsilon$ small enough, $T_\brep$ is strictly positive. We can then apply Proposition \ref{Prop:Cut} to $T_\brep$ and get a projection $Q$ such that $T_1 := Q T_\brep Q$ is nonzero, positive and gauge-invariant.

Applying Proposition \ref{Prop:Aperiod} to $T_1$, we recover $A$ s.t. $A^* T_1 A = 1 = A^* Q T_\brep Q A$. Hence,
$$ \| A^* Q T Q A - 1 \| = \| A^* Q T Q A - A^* Q T_\brep Q A \| \leqslant \| A \|^2 \| T - T_\brep \|,$$
because $\| Q \| = 1$. To conclude, we need to find a lower bound for $\| T_1 \|$ in order to use the estimate on $\| A\|$.
$$ \| \E_{S^1}(T_\brep ) - \E_{S^1}(T) \| \leqslant \| T_\brep  - T \| \leqslant \varepsilon $$
implies $\| T_1 \| = \| \E_{S^1}(T_\brep ) \| \geqslant 1 - \varepsilon$ since $\| \E_{S^1}(T) \| = 1$. Thus Proposition \ref{Prop:Aperiod} yields $\| A \|^2 \leqslant 2$ for $\varepsilon$ smaller than $1/2$ and we get 
$$ \| (Q A)^* T Q A - 1 \| \leqslant 2 \varepsilon .$$
Therefore, the algebra $\Falg $ is PI.
\end{proof}

\begin{remarque}
\label{Rk:Zacharias}
We proved Theorem \ref{Thm:PI1} because we could not find a choice of ``coefficient algebra'' for $\Falg $ which would imply Property ($\ast$) of \cite{QfreeAutZacharias}.

\smallbreak

Let us illustrate, in the case of the Cuntz algebra, that this choice of coefficient algebra is crucial to apply Property ($\ast$). Indeed, the Cuntz algebra $\CAlg_d$ is a Pimsner algebra (Example (1) p.192 of \cite{PiCrG}). It can be realised either with coefficient algebra $A = \C$ (and $E = \C^d$) or with $A = \InvJ$ (and $E = (\CAlg_d)^{(1)}$). 
\begin{itemize}
\item
For $A = \C$, Property ($\ast$) says that $\xi$ and $\eta$ are two orthonormal elements of $\C^d$ -- thus this condition is satisfied. 
\item
For $A = \InvJ$, we can always consider $c = \xi \eta^*$ and evaluate: $\langle \xi, c \eta \rangle = \xi^* c \eta = \xi^* \xi \eta^* \eta = 1$ since both $\xi$ and $\eta$ are isometries. Thus Property ($\ast$) is \emph{not} satisfied.
\end{itemize}
However, in some sense Property ($\ast$) and Condition \ref{Cond:Dim} are similar, since they both ensure that $E$ is ``large enough''.
\end{remarque}

\section{Computation of $K$-groups}
\label{Sec:CompKth}

In this section, we start by a quick review of (Cuntz-)Pimsner algebras. We then prove that under Condition \ref{Cond:Entrelac}, the fixed point algebras are Pimsner algebras. Applying Theorem 4.9 of \cite{PiCrG}, we compute the $K$-theory of $\Falg $ -- which only depends on $\CQG $ up to $R^+$-isomorphism.

\smallbreak

Very analogous results for $\InvJ^\brep $ were first stated in the case of (ordinary)  compact groups by Antony Wassermann in his PhD thesis \cite{TheseWassermann} (see in particular III.3.(iii) p.103 and III.7. p.157). He also sketched a proof of the results for $\Falg$. Here we give complete proofs of the results in the CQG setting.

We use the book \cite{Black} as a general reference, especially for Hilbert modules and $C^*$-correspondences. We render Def. II.7.1.1 p.137 
therein by:
\begin{definition}
A right \emph{Hilbert module} $E$ over a $C^*$-algebra $A$ is a Banach space endowed with a right action of $A$ together with a map $\langle \xi, \eta \rangle _A \colon E \times E \to A$ which satisfies certain conditions (``$A$-valued inner product'').
\end{definition}
If $E$ is a Hilbert module over a $C^*$-algebra $A$, we say that a function $T \colon E \to E$ is an \emph{adjointable operator} if we can find a function $T^* \colon E \to E$ s.t. for all $\xi, \eta \in E$,
$$ \langle T \xi, \eta \rangle _A = \langle \xi, T^* \eta \rangle _A .$$
The set of adjointable operators on $E$, denoted $\OpLin(E)$, is a $C^*$-algebra (\cite{Black} II.7.2.1). 
Inside this $C^*$-algebra lies the ideal $\OpCpct (E)$ (\cite{Black} II.7.2.4) 
of ``compact operators on $E$'', defined as the closed linear span of operators $\Theta_{\eta, \zeta}(\xi) = \eta \langle \zeta, \xi \rangle _A$. 

 In this article, we need a stronger structure called \emph{$C^*$-correspondence} (\cite{Black} II.7.4.4 -- also called ``Hilbert bimodule'' in \cite{PiCrG}).
\begin{definition}
\label{Def:C*corr}
A \emph{$C^*$-correspondence} $E$ over a $C^*$-algebra $A$ is a right Hilbert module $E$ over $A$, together with a left-action of $A$, \textsl{i.e.} a $*$-morphism $\phi \colon A \to \OpLin(E)$.

When the algebra $A$ is unital, we further require that $\phi(1_A) = \id_E$. 
\end{definition}
In the case of a finitely generated (f.g.) right $A$-module $E$, we call \emph{frame} a finite family $(\xi_i) \in E$ s.t. for all $\xi \in E$,
\begin{equation}
\label{Eqn:Frame}
 \sum_{i} \xi_i \langle \xi_i, \xi \rangle _A = \xi .
\end{equation}
\begin{remarque}
\label{Rem:CompRep}
In Def. 3.8 of \cite{PiCrG}, an ideal $I$ is defined by %
$ 
I := \phi^{-1}(\OpCpct (E))
$. This ideal is used for the six-term exact sequence \eqref{Eqn:6Terms} of the present article. When the algebra $A$ is unital and a frame exists, a consequence of \eqref{Eqn:Frame} and our Definition \ref{Def:C*corr} is that $1_A \in I$ and thus $I = A$.
\end{remarque}
Here, we use the characterisation of Pimsner algebras given by Theorem 3.12 in \cite{PiCrG} as definition. We further adapt this definition to the case we consider here, \textsl{i.e.} the $C^*$-algebra $A$ is unital and the $C^*$-correspondence possesses a (finite) frame (see Remark 3.9 p.202 in \cite{PiCrG}).

\begin{definition}
\label{Def:CompRep}
Given a $C^*$-correspondence $E$ over $A$ with a frame $(\xi_i)$, a \emph{compatible representation} of $E$ on a $C^*$-algebra $B$ is a $*$-homomorphism $\pi \colon A \to B$ and a linear map $t \colon E \to B$ s.t.
\begin{enumerate}[label=(\roman*)]
\item \label{Pt:1}
$t(\xi)^* t(\eta) = \pi( \langle \xi, \eta \rangle _A) $ for all $\xi, \eta \in E$,
\item \label{Pt:2}
 $\pi(a) t(\xi) = t( \phi(a) \xi) $ for all $a \in A$, $\xi \in E$.
\item \label{Pt:3}
using the frame $(\xi_i)$, $ \sum_{i} t(\xi_i) t(\xi_i)^* = \pi(1)$.
\end{enumerate}
\end{definition}

\begin{remarque}
Regarding this definition, we observe:
\begin{itemize}
\item
Conditions \ref{Pt:1} and \ref{Pt:2} imply that $t(\xi a) = t(\xi) \pi(a)$ (see the discussion after Def. 2.1 in \cite{Katsura04}).
\item%
A quick computation shows that if condition \ref{Pt:3} is satisfied for one frame, then it holds for all frames.
\end{itemize}
\end{remarque}

\begin{definition}
The \emph{(Cuntz-)Pimsner} algebra $\CAlg_E$ generated by the $C^*$-correspondence $E$ over $A$ is the universal $C^*$-algebra generated by compatible representations of $E$. 

In other words, there is a compatible representation of $E$ on $\CAlg_E$ and any compatible representation $(\pi, t)$ of $E$ on a $C^*$-algebra $B$ induces a unique $C^*$-algebra morphism $A \to B$ which factors $(\pi, t)$.
\end{definition}
To define such a universal algebra, all generators must have finite bounds on their norms (\cite{Black}, II.8.3). For elements $a \in A$, the property is obvious since $\pi$ is contracting. For any $\xi \in E$, we have 
$$
\| t(\xi) \|^2 = \| t(\xi)^* t(\xi) \| = \| \pi(\langle \xi, \xi \rangle _A) \| \leqslant \| \langle \xi, \xi \rangle _A \| .
$$

\begin{prop}
\label{Prop:PimsnerAlg}
If $\brep$ satisfies Condition \ref{Cond:Entrelac}, then 
\begin{itemize}
\item
we can find an isometry $\nu \in \Hh^{N+1} (\Hh^*)^N \cap \Falg $,
\item
 $\Falg $ is isomorphic to the Pimsner algebra $\CAlg_E$ generated by the natural $C^*$-cor\-respon\-dence $E = \nu^* \InvJ^\brep $ over $A = \InvJ^\brep $.
\end{itemize}
\end{prop}

\begin{proof}
Using Condition \ref{Cond:Entrelac}, we can find a norm preserving map $\psi \colon \Hh^{\ell} \to \Hh^{\ell +1}$ intertwining the $\CQG $-action. In particular, $\psi^* \psi = \id_{\Hh^{\ell}}$. From Lemma  \ref{Lem:InterpEntrelac}, we see that we can interpret $\psi$ as an element $\nu \in \Hh^{N+1} (\Hh^*)^N \cap \Falg $. Moreover, $\nu^* \nu$ corresponds to $\psi^* \psi$ and thus $\nu^* \nu = 1$.

\smallbreak

We now introduce $E = \nu^* \InvJ^\brep $ and show that this is a $C^*$-correspondence over $\InvJ^\brep $. It is clear that $E$ is a right $\InvJ^\brep $-module. We can identify this set with all elements $\xi$ in $\Falg $ which satisfy:
\begin{equation}
\label{Eqn:EltHomDeg1}
 \gamma_z(\xi) = \overline{z} \xi .
\end{equation}
Indeed, this relation holds for all elements in $E$. Inversely, if $\xi \in \Falg $ satisfies \eqref{Eqn:EltHomDeg1}, then $\xi = \nu^* (\nu \xi)$ and $\gamma_z(\nu \xi) = \gamma_z(\nu) \gamma_z(\xi) = z \nu \overline{z} \xi = \nu \xi$ thus $\nu \xi \in \InvJ^\brep $. By very much the same method, we show that all gauge-homogeneous elements $T \in \Falg $ can be written either $(\nu^*)^k a$ or $a \nu^k$ for some $k \in \N$ and $a \in \InvJ^\brep $. 

\smallbreak

We define the scalar product on $E$ by the natural formula:
$$ \langle \xi, \eta \rangle_{\InvJ^\brep } = \xi^* \eta .$$
We see that $\xi^* \eta$ is indeed in $\InvJ^\brep $ by applying the gauge action. It is readily checked that this scalar product satisfies all conditions for a Hilbert module (see II.7 in \cite{Black}). Note that this module has a single generator $\nu^* $ which also defines a frame of $E$: for all $\xi \in E$,
\begin{equation}
\label{Eqn:DefPiScal}
\nu^* \langle \nu^* , \xi \rangle _{\InvJ^\brep } = \nu^* \nu \xi = \xi .
\end{equation}
We further define a $C^*$-correspondence structure on $E$ using the multiplication in $\Falg $. Indeed, for any $a \in \InvJ^\brep $, $\xi \in \nu^* \InvJ^\brep $, $ a \xi $ is homogeneous of total gauge $-1$.

We now prove that there is a compatible representation of $E$ on $\Falg $ and then apply the gauge-invariant uniqueness Theorem 6.4 of \cite{Katsura04} to prove that $\Falg $ is isomorphic to the Pimsner algebra $\CAlg_E$.

This is purely formal: the natural $C^*$-algebra morphism $\pi \colon \InvJ^\brep \to \Falg $ is injective. For the linear map $t$, we set $t(\nu^* a) = \nu^* a$. It is clear from \eqref{Eqn:DefPiScal} that we get a compatible representation and thus a map $\CAlg_E \to \Falg $. This representation clearly admits a gauge action (in the sense of Def. 5.6 of \cite{Katsura04}), which is just $\gamma$. Therefore, Theorem 6.4 of \cite{Katsura04} applies, proving that $\CAlg_E \to \Falg $ is injective.

\smallbreak

To show surjectivity, we prove that all algebraic elements are reached and rely on Lemma \ref{Lem:AlgApprox}. Algebraic elements are finite sums of homogeneous elements and the latter are of the form $(\nu^*)^k a$ or $a \nu^k$ for some $k \in \N$ and $a \in \InvJ^\brep $. But these are reached by $\big( t(\nu)^* \big)^k \pi(a)$ and $\pi(a)t(\nu)^k$ respectively. Therefore, the image of $\CAlg_E \to \Falg $ is dense, proving surjectivity.
\end{proof}

To compute the $K$-theory of $\Falg $, we start by describing the algebra $\InvJ^\brep $. By Lemma \ref{Lem:AlgApprox}, this is an AF-algebra determined by $\lim_{\to } \InvJ^{\brep, \ell}$ and the induced Bratteli's diagram (see for instance \cite{RordamStormer} p.13). In particular, $K_1(\InvJ^\brep) = 0$. To evaluate $K_0(\InvJ^\brep )$, fix $\ell$ and consider the decomposition into irreducible representations of $\brep^\ell$:
\begin{equation}
\label{Eqn:brepEll}
\brep ^\ell \simeq \bigoplus d_t (t) ,
\end{equation}
where all irreducible representations $(t)$ are in $\CatInd$ (by definition) and only a finite number of $d_t$ are nonzero. From Lemma \ref{Lem:InterpEntrelac} together with the Schur's lemma for representations of $\CQG $, we get that $\InvJ^{\brep , \ell}$ is isomorphic to 
\begin{equation}
\label{Eqn:InvJEll}
 \InvJ^{\brep , \ell} \simeq \bigoplus M_{d_t} (\C) .
\end{equation}
Hence, the different matrix components of $\InvJ^\brep $ are determined by the irreducible representations appearing in $\alpha^\ell$. The connecting maps $\varphi_\ell \colon \InvJ^{\brep, \ell} \to \InvJ^{\brep, \ell+1}$ are fully determined by the fusion rules $(t) \otimes \brep $ for $(t)$ appearing in \eqref{Eqn:brepEll}. Indeed, if 
$$ (t) \otimes \brep = m_1 (\tau_1) \oplus \cdots \oplus m_p (\tau_p) $$
then the multiplicity of $\varphi_\ell$ between the $(t)$ component of $\InvJ^{\brep, \ell}$ and the $(\tau_p)$ component of $\InvJ^{\brep, \ell+1}$ is $m_p$. For a more concrete illustration, refer to Section \ref{SSec:Example} of the present paper. Since $K_0$ is continuous, 
we recover $K_0(\InvJ^\brep ) = \lim_{\to } K_0(\InvJ^{\brep , \ell})$ and thus

\begin{prop}
\label{Prop:KthInvJ}
The $K$-theory of $\InvJ^\brep $ is fully determined by the fusion rules on $\CQG $. In other words, it only depends on $\CQG $ up to $R^+$-isomorphism (Def. \ref{Def:IsomRepCat}).
\end{prop}

\begin{remarque}
\label{Rem:brep}
Of course, the $K$-theory of $\InvJ^\brep $ also depends on the choice of $\brep $. But if $\CQG$ and $\CQG'$ are $R^+$-isomorphic, then $\brep $ is sent by the isomorphism to some representation $\brep '$ of $\CQG '$, giving a meaning to the above proposition.
\end{remarque}

We have $K_0( \InvJ^{\brep , \ell}) \simeq \bigoplus_{t \in \CatInd ^{\ell}} \Z$ where $\CatInd ^{\ell}$ is the set of irreducible representations appearing in \eqref{Eqn:brepEll} with nonzero multiplicity. Hence, the elements of $K_0(\InvJ^{\brep , \ell})$ can be represented as formal sums $\sum_{t \in \CatInd ^{\ell}} n_t (t)$. It is then natural to consider $\Z[\CatInd]$, the formal sums on $\CatInd$ with integer coefficients. It is obvious that if $\beta$ and $\gamma$ are irreducible representations in $\CatInd$, then the irreducible representations appearing in the decomposition of $\beta \otimes \gamma$ are also in $\CatInd$. Thus there is a product in $\CatInd$ induced by the tensor product and $\Z[\CatInd]$ inherits a ring structure from it.

\smallbreak

If the ring $\Z[\CatInd ]$ is commutative -- it is the case for compact groups and their $R^+$-deformations -- further simplifications arise, involving the localised ring $\Z[\CatInd ]\left[ \frac{1}{\brep } \right]$. The definition of this set as inductive limit, together with the expression of the connecting maps $K_0(\varphi_\ell) \colon K_0(\InvJ^{\brep, \ell}) \to K_0(\InvJ^{\brep, \ell+1})$:
$$ K_0(\varphi_{\ell})\left(\sum_{} n_t (t) \right) = \sum_{}n_t \big( (t) \otimes \brep \big) $$
show that $K_0(\InvJ^\brep )$ can be realised as a submodule of $\Z[\CatInd ]\left[ \frac{1}{\brep } \right]$. More precisely (compare \cite{TheseWassermann}, p.103):
\begin{remarque}
\label{Rem:Fraction}
$K_0(\InvJ^\brep )$ is the set of all fractions $\left( \sum_{} n_t (t) \right) / \brep ^\ell$ where all $(t)$ in the sum appear with nonzero multiplicity in \eqref{Eqn:brepEll}.
\end{remarque}

Going back to the general case, we have the following:
\begin{thm}
\label{Thm:KthFalg}
If $\brep $ satisfies Property \ref{Cond:Entrelac}, the $K$-theory of $\Falg$ is
\begin{align}
\label{Eqn:KthFalg}
K_0(\Falg ) &= \coker(\id - [E]_*)
&
K_1(\Falg ) &= \ker(\id - [E]_*),
\end{align}
where  $[E]_*$ is the map induced by the $C^*$-correspondence $E$ using Kasparov product. If moreover $\Z[\CatInd ]$ is an integral domain and $\brep \neq \TrivRep$, then $K_1(\Falg ) = 0$.

\smallbreak

In any case, $K_*(\Falg )$ only depends on $\CQG $ up to $R^+$-isomorphism, in the sense of Remark \ref{Rem:brep}.
\end{thm}

\begin{proof}
Proposition \ref{Prop:PimsnerAlg} (and its proof) ensures that we can apply Theorem 4.9 of \cite{PiCrG} to $\Falg $. Consider the first six-term exact sequence obtained from this theorem. We can substitute $B \rightsquigarrow \C$ and $I \rightsquigarrow \InvJ^\brep $ -- because of Remark \ref{Rem:CompRep} -- and find:
\begin{equation}
\label{Eqn:6Terms}
\xymatrix@C=1.5cm{
K_0(\InvJ^\brep ) \ar[r]^-{\id -  [E]_*} & K_0(\InvJ^\brep ) \ar[r]^-{\iota_*} & K_0(\Falg ) \ar[d]^{\partial } \\
K_1(\Falg ) \ar[u]^-\partial  & K_1(\InvJ^\brep ) \ar[l]_-{\iota_*} & K_1(\InvJ^\brep ) \ar[l]_-{\id -  [E]_*} .\\
}
\end{equation}
In the above, $[E]_*$ is the left Kasparov multiplication by the $KK$-theory class induced by the $C^*$-correspondence $[E] \in KK(\InvJ^\brep , \InvJ^\brep )$. 

Since $K_1(\InvJ^\brep ) = 0$, this sequence can actually be written:
$$
0 \to K_1(\Falg ) \to K_0(\InvJ^\brep ) \xrightarrow{1 - [E]_*} K_0(\InvJ^\brep ) \to K_0(\Falg ) \to 0,
$$
yielding the equalities \eqref{Eqn:KthFalg}.

To complete the computation, we want to describe the action of $[E]_*$ on $K_0(\InvJ^{\brep , \ell})$ for $\ell \geqslant N$ -- where $N$ is the integer of Proposition \ref{Prop:PimsnerAlg}. Explicitly, we interpret (the positive part of) $K_0(\InvJ^\brep )$ as f.g. right $\InvJ^\brep $-Hilbert modules $M_{\InvJ^\brep}$: if $(\xi_i)$ is a frame for $M_{\InvJ^\brep}$, then this module corresponds to the projection $e = (\langle \xi_i, \xi_j \rangle _{\InvJ^\brep})_{i,j} \in M_n(\InvJ^\brep)$. In the case of $E = \nu^* \InvJ^\brep  $, Equation \eqref{Eqn:DefPiScal} proves that $\nu^* $ is a frame.

\smallbreak

By definition, $[E]_*(M_{\InvJ^\brep}) = M \otimes _{\InvJ^\brep} E$ and it is easy to check that $(\xi_i \otimes \nu^*)$ is a frame for $M \otimes _{\InvJ^\brep} E$. 
Hence, the projection $e$ is sent to 
\begin{multline*}
f = \Big(\langle \xi_i \otimes \nu^*, \xi_j \otimes \nu^* \rangle_B \Big)_{i,j} = \Big(\langle \nu^*, \langle \xi_i , \xi_j \rangle _B \nu^* \rangle _B \Big)_{i,j} =\\
\Big(\nu e_{i,j} \nu^* \Big)_{i,j} \in M_n(\InvJ^{\brep , \ell+1}).
\end{multline*}
This element acts naturally on $\C^n \otimes \Hh^{\ell + 1}$. We know from Condition \ref{Cond:Entrelac} that there is an injection $\psi \colon \Hh^{\ell} \to \Hh^{\ell +1}$ which preserves the $\CQG $-action. $f$ is just the image of $e$ by this injection. In particular, $f$ decomposes into precisely the same number of components as $e$, with the same multiplicities. In other words, $[E]_*$ sends $\sum_{} n_t (t) \in K_0(\InvJ^{\brep , \ell})$ to $\sum_{} n_t (t) \in K_0(\InvJ^{\brep , \ell+1})$. Such a map exists because of Condition \ref{Cond:Entrelac}: indeed, if the irreducible representation $(t)$ appears in $\brep ^{\ell}$, then it also appears in $\brep ^{\ell +1}$.

\medbreak

The kernel of $1 - [E]_*$ is characterised as all elements $x \in \Z[\CatInd ]$ such that $x = [E]_* x $ or equivalently $x \brep = x$ -- where we chose a realisation of $x$ inside some $\InvJ^{\brep , \ell}$ for $\ell$ large enough. When $\Z[\CatInd ]$ is an integral domain -- this is realised for compact connected Lie groups (see Corollary 2.8 p.167 of \cite{BrockertomDieck}) and their $R^+$-deformations -- if $\brep \neq \TrivRep$ (trivial rep.) the only solution to this equation is $x = 0$, hence $K_1(\Falg ) = 0$.
\end{proof}

\section{Main results}
\label{Sec:Main}

We refer to \cite{RordamStormer} Def. 4.3.1 for the following:
\begin{definition}
A \emph{Kirchberg algebra} is a PI, simple, nuclear and separable $C^*$-algebra.
\end{definition}

Under some hypotheses, these algebras are fully classified by their $K$-theory, according to the Kirchberg-Phillips classification theorem. To make this assumption precise, we need the following definition, borrowed from \cite{Black}, V.1.5.4: 
\begin{definition}
\label{Def:UCTclass}
The (large) \emph{bootstrap class} or Universal Coefficient Theorem (\emph{UCT}) \emph{class} is the smallest class $\UCTclass$ of separable nuclear $C^*$-algebras s.t.
\begin{enumerate}[label=(\roman*)]
\item
$\C \in \UCTclass$;
\item
$\UCTclass$ is closed under inductive limit; 
\item
\label{Pt:UCT3}
if $0 \to J \to A \to A/J \to 0$ is an exact sequence, and two $C^*$-algebras $J, A$ or $A/J$ are in $\UCTclass$, then so is the third; 
\item
$\UCTclass$ is closed under $KK$-equivalence.
\end{enumerate}
\end{definition}
To avoid technicalities, we treat the notion of ``$KK$-equivalence'' (also called ``$K$-equivalence'' in \cite{PiCrG}) as a ``black box''. We refer the interested reader to \cite{Black}, Sections V.1.4 and V.1.5 and references therein.

\smallbreak

We can now state Kirchberg-Phillips classification theorem, which first appeared in \cite{ClassThmKirchberg,ClassThmKirchbergKP} (see also Theorem 8.4.1 (iv) p.128 
of \cite{RordamStormer}):
\begin{thm}
\label{Thm:ClassKirchberg}
Let $A$ and $B$ be unital Kirchberg algebras in the UCT class $\UCTclass$. $A$ and $B$ are isomorphic if and only if there are isomorphisms 
\begin{align*}
\alpha_0 &\colon K_0(A) \to K_0(B)
&
\alpha_1 &\colon K_1(A) \to K_1(B)
\end{align*}
with $\alpha_0([1_A]_0) = [1_B]_0$. For each such pair of isomorphisms there is an isomorphism $\varphi \colon A \to B$ with $K_0(\varphi) = \alpha_0$ and $K_1(\varphi) = \alpha_1$.
\end{thm}

Bringing together Theorem \ref{Thm:PI1} and Proposition \ref{Prop:PimsnerAlg}, we get the following proposition:
\begin{prop}
\label{Prop:DualInvA}
Let $\brep $ be a representation of $\CQG $ which satisfies Conditions \ref{Cond:Contra}, \ref{Cond:Dim} and \ref{Cond:Entrelac}. The fixed point algebra $\Falg $ is both a Kirchberg algebra in the UCT class $\UCTclass$ and a Pimsner algebra.
\end{prop}

\begin{proof}
We already know from Theorem \ref{Thm:PI1} that $\Falg $ is a simple and purely infinite algebra. It only remains to prove that it is separable, nuclear and in $\UCTclass$.

It is an immediate consequence of Lemma \ref{Lem:AlgApprox} that $\Falg $ is separable. Nuclearity comes from Corollary 7.4 of \cite{Katsura04}, which states that in the case of a Pimsner algebra, if the coefficient algebra (here $\InvJ^\brep$) is nuclear, then so is the Pimsner algebra. In our case, Lemma \ref{Lem:AlgApprox} ensures that $\InvJ^\brep$ is AF and thus nuclear -- this is a consequence of Proposition 2.1.2 of \cite{RordamStormer}. 

Concerning $\UCTclass$, we have a short exact sequence of $C^*$-algebras, similar to Pimsner-Voiculescu exact sequence (see Section 8 of \cite{Katsura04} p.386 or Theorem 3.13 of \cite{PiCrG}):
$$ 0 \to C^\brep \to \Toep_E^\brep \to \Falg  \to 0 .$$
Observe, regarding the algebras in this extension, that:
\begin{itemize}
\item
Theorem 4.4 of  \cite{PiCrG} shows that $\Toep_E^\brep$ is $KK$-equivalent to the coefficient algebra of the Pimsner algebra, here $\InvJ^\brep $. 
\item
Lemma 4.7 of \cite{PiCrG} proves a $KK$-equivalence between $C^\brep$ and his ideal $I$, which we identified with $\InvJ^\brep$ in Remark \ref{Rem:CompRep}.
\end{itemize}
Moreover, $\InvJ^\brep$ is AF and thus in $\UCTclass$. 
In this situation, point \ref{Pt:UCT3} of Def. \ref{Def:UCTclass} above ensures that $\Falg $ is in $\UCTclass$.
\end{proof}

Proposition \ref{Prop:DualInvA} enables us to apply Theorem \ref{Thm:ClassKirchberg} to $\Falg$ and thus get our main theorem:
\begin{thm}
\label{Thm:Main}
Let $\brep $ be a representation of $\CQG $. If $\brep $ satisfies Properties \ref{Cond:Contra}, \ref{Cond:Entrelac} and \ref{Cond:Entrelac}, then up to $C^*$-isomorphism, the fixed point algebra $\Falg $ only depends on the $R^+$-isomorphism class of $\CQG $, in the sense of Remark \ref{Rem:brep}.
\end{thm}

\begin{proof}
The result is immediate by combining Theorem \ref{Thm:ClassKirchberg}, Proposition \ref{Prop:DualInvA} and Theorem \ref{Thm:KthFalg}.
\end{proof}

\section{Discussion and examples}
\label{Sec:Discussion}

We comment on our Conditions \ref{Cond:Contra}, \ref{Cond:Dim} and \ref{Cond:Entrelac}.
\begin{itemize}
\item
The first condition is non-trivial. Indeed some CQG do not satisfy Condition \ref{Cond:Contra}: consider $\CQG = U(1)$ and its representation $\mu \colon  z \mapsto z$. Clearly, in $T_\mu$ we get all the representations $z \mapsto z^n$ for $n >1$ but we do not recover the representation $z \mapsto z^{-1}$ which would satisfy \ref{Cond:Contra}.
\item
Hypothesis \ref{Cond:Dim} is a somewhat mild condition on $\brep $. Indeed, it is only violated when $\brep $ is a single irreducible representation.
\end{itemize}
Unfortunately, in view of \cite{CuntzAlgDR,EndomDR}, the most interesting cases are the natural representations of $SU(N)$, which are irreducible. In the following section, we will see how to bypass the problem and apply our theory to natural representations of $\CQG = SU(N)$. We will see that in this context, condition \ref{Cond:Entrelac} is also satisfied.

\medbreak

A crucial tool in applying our theory is the notion of \emph{chain group} that we adapt from \cite{HilbertCSystBaumgLledo} (see Theorem 5.5), where it was studied for (ordinary) compact groups:
\begin{definition}
The \emph{chain group} $\ChainGp(\CQG )$ is defined as the set of equivalence classes $[t]$ of irreducible representations of $\CQG $ under the relation $\sim$, where $(t)\sim (t')$ if and only if we can find irreducible representations $(\tau_1), \ldots , (\tau_n)$ s.t. $(t)$ and $(t')$ appear in the decomposition of the tensor product $(\tau_1) \otimes \cdots \otimes (\tau_n)$ into irreducible components. The tensor product induces a group structure on this set.
\end{definition}
The proofs of the original article adapt seamlessly to the case of CQG. The only difference is that in general $\ChainGp(\CQG )$ is no longer Abelian. By construction, $\ChainGp(\CQG )$ only depends on the fusion rules of $\CQG $, \textsl{i.e.} on its $R^+$-isomorphism class.

\medbreak

If $\brep $ is a single irreducible representation and \ref{Cond:Contra} is satisfied, then the \emph{semigroup} generated in $\ChainGp(\CQG )$ by $[\brep ]$ is actually a group. This subgroup of $\ChainGp(\CQG )$ is then  \emph{finite} and \emph{Abelian}. In the case of $\CQG = U(1)$, this property fails. 

However, we can find two large classes on which \ref{Cond:Contra} is satisfied. First we have semisimple Lie groups:
\begin{prop}
\label{Prop:SSLieGp}
If $\CQG $ is a semisimple Lie group (or a $R^+$-deformation thereof), then it satisfies \ref{Cond:Contra} for any irreducible representation $\brep $.
\end{prop}

\begin{remarque}
\label{Rem:IddCG}
Under the hypothesis of the proposition, the chain group can be identified with the character group of the center of $\CQG $. The identification is defined by the restriction of the irreducible representation to the center of $\CQG $ (see \cite{HilbertCSystBaumgLledo}, (5.2) and Theorem 5.5).
\end{remarque}

\begin{proof}
Without loss of generality, we assume that $\CQG $ is a semisimple Lie group. We are actually going to prove that for any irreducible representation $\beta$ of $\CQG $, the tensor product $\beta^{\otimes n}$ contains the trivial representation $\TrivRep$. Since $\CQG $ is a semisimple Lie group, it has a unique dimension $1$ representation. If $\beta$ acts on $\Hh$, $\beta^{\otimes n}$ induces a representation on $\bigwedge^n \Hh \simeq \C$. This concludes the proof.
\end{proof}

The second large class which satisfies Condition \ref{Cond:Contra} is finite groups\footnote{We thank J.-F. Planchat for drawing our attention to this fact.}:
\begin{prop}
\label{Prop:FiniteGp}
If $\CQG $ is a finite group (or a $R^+$-deformation thereof), then it satisfies \ref{Cond:Contra} for any irreducible representation $\brep $.
\end{prop}

\begin{proof}
The proof is similar to that of Proposition \ref{Prop:SSLieGp}: we start from any irreducible representation $\beta$ and prove that some iterated tensor product $\beta^{\otimes n}$ contains the trivial representation. 

\smallbreak

If $\beta$ acts on a $d$-dimensional Hilbert space $\Hh$, then it induces a dimension $1$ representation on $\bigwedge^n \Hh$ -- which is included in $\beta^{\otimes n}$. Since this representation $\bigwedge^n \beta$ has dimension $1$, for any $g \in \CQG $, we have:
$$ (\wedge^n \beta(g))^{\otimes k} = \wedge^n \beta(g^k) .$$
If we take $k = N$, the order of the group $\CQG $, then $(\bigwedge^n \beta)^{\otimes N}$ is the trivial representation. The result follows.
\end{proof}

If all irreducible representations in $\brep$ have the same class $[\brep ]$ in $\ChainGp(\CQG )$, we shall see in the next subsection that $[\brep ] \neq e$ is an obstruction to Condition \ref{Cond:Entrelac}. The following proposition enables us to get around this difficulty:
\begin{prop}
\label{Prop:ResSUN}
Let $\brep$ be a representation of a CQG $\CQG$. 
\begin{enumerate}[label=(\roman*)]
\item
For any $M \geqslant 1$, there is an injection $\Falg[\brep ^M] \to \Falg$.
\item
If all irreducible representations in $\brep$ have the same nontrivial class in $\ChainGp(\CQG )$, which we denote $[\brep ] \neq e$, and $M$ is the order of $[\brep ]$ inside $\ChainGp(\CQG )$, then $\Falg[\brep ^M] \to \Falg$ is surjective.
\end{enumerate} 
\end{prop}

\begin{proof}
Let us prove the injectivity: denote $d \in \N$ the dimension of the Hilbert space $\Hh$ underlying the representation $\brep$. By definition, $\Falg  \subseteq \CAlg_d$ and $\Falg[\brep^M] \subseteq \CAlg_{d^M}$. There is a $C^*$-algebra morphism $\CAlg_{d^M} \to \CAlg_d$ defined on generators by:
$$ S_{K} \to S_{k_0} S_{k_1} S_{k_2} \cdots S_{k_{M-1}}$$
where $K \in \{1, 2, \ldots , d^M\}$ and $k_0, k_1, \ldots , k_{M-1}\in \{ 1,2, \ldots , d\}$ are uniquely determined by $K = k_0 + k_1 d + k_2 d^2 + \cdots + k_{M-1} d^{M-1}$. Indeed, using the properties of the generators of $\CAlg_d$:
$$
(S_{k_0} S_{k_1} S_{k_2} \cdots S_{k_{M-1}})^* S_{k_0'} S_{k_1'} S_{k_2'} \cdots S_{k_{M-1}'}= \delta_{k_0, k_0'} \delta_{k_1, k_1'}\delta_{k_2, k_2'}\cdots \delta_{k_{M-1}, k_{M-1}'} 1
$$
and 
$$
\sum_{k_0, k_1, k_2, \ldots , k_{M-1}} S_{k_0} S_{k_1} S_{k_2} \cdots S_{k_{M-1}} (S_{k_0} S_{k_1} S_{k_2} \cdots S_{k_{M-1}})^*= 1.
$$
The universal property of $\CAlg_{d^M}$ then yields an injective morphism $\CAlg_{d^M} \to \CAlg_d$. Next step, we restrict and corestrict to a morphism $\Falg[\brep^M] \to \Falg $.

To prove that the range of $\Falg[\brep^M] \to \CAlg_d$ is included in $\Falg $, we rely on Lemma \ref{Lem:AlgApprox}. It thus suffices to prove that the image of any algebraic element $T \in \Falg[\brep^M]$ is in $\Falg $. Without loss of generality, we can assume that $T$ is gauge-homogeneous of total gauge $k$. We can then find integers $k_0, k_1$ s.t.
$$ T \in \Hh_M^{k_0} (\Hh_M^*)^{k_1} $$
where the difference $k_0 - k_1 = k$ is the total gauge of $T$ in $\Falg[\brep^M]$ and $\Hh_M$ is the $M$-fold tensor product (over $\C$) of Hilbert spaces $\Hh_M := \Hh^{\otimes M}$. This is the Hilbert space on which $\brep ^M$ is represented. In this setting, $T$ is in $\Falg[\brep ^M]$ if and only if it intertwines $ \Hh_M^{k_1}$ endowed with $(\brep^M)^{k_1}$ and $ \Hh_M^{k_0}$ with $(\brep^M)^{k_0}$ (Lemma \ref{Lem:InterpEntrelac}). But identifying $(\Hh_M)^{k_i}$ with $\Hh^{M k_i}$ and $(\brep^M)^{k_i}$ with $\brep^{M k_i}$, it is clear that $T$ can be seen as element of $\Falg $ -- thus the induced morphism $\Falg[\brep^M] \to \Falg $ is well defined. Moreover, it is injective since $\CAlg_{d^M} \to \CAlg_d$ is.

\smallbreak

Note that the choice of the morphism $\CAlg_{d^M} \to \CAlg_d$ dictates the explicit form of the identification of $\Hh_M$ with $\Hh^{\otimes M}$.

\medbreak

Further assuming that all irreducible representations in $\brep$ have the same class $[\brep ] \neq e$ in $\ChainGp(\CQG )$ and that the order of $[\brep ]$ is $M$, we prove surjectivity by using the chain group. It suffices to show that any algebraic element $T$ of $\Falg $ which is homogeneous with total gauge $k$ is reached. For such a $T$, we can always find $k_0$ and $k_1$ large enough to have
$$ T \in \Hh^{k_0} (\Hh^*)^{k_1} $$
with $k = k_0 - k_1$, total gauge of $T$. If such a nonzero intertwiner exists, then in particular $[\brep ]^{k_0} = [\brep ]^{k_1}$ \textsl{i.e.} $[\brep ]^{k_0 - k_1} = e$ in $\ChainGp(\CQG )$. By definition of $M$, this is equivalent to $k = k_0 - k_1$ being a multiple of $M$. Upon iterating the inclusion 
$$ \Hh^{k_0} (\Hh^*)^{k_1} \inj \Hh^{k_0+1} (\Hh^*)^{k_1+1}, $$
we can further assume that $k_0$ and thus $k_1$ are multiples of $M$. In this case, it is clear that $T \in \Falg $ can be lifted to $T' \in \Falg[\brep ^M]$.
\end{proof}
A consequence of the previous proposition is (for $N \geqslant 2$):
\begin{prop}
\label{Prop:RedSUN}
If $\CQG = SU(N)$ (or a $R^+$-deformation thereof) and $\brep = \nu$ is the natural representation of $\CQG $ on $\C^N$, then $\Falg$ is both a unital Kirchberg algebra in the UCT class $\UCTclass$ and a Pimsner algebra.
\end{prop}

\begin{remarque}
This proposition complements the explicit description by generators of $\Falg$ provided for this case by Paolucci (see \cite{CoactionCuntzAlgPaolucci}, Lemma 7).
\end{remarque}

\begin{proof}
 Proposition \ref{Prop:SSLieGp} ensures that Condition \ref{Cond:Contra} is satisfied. If $\nu$ is the natural representation of $\CQG = SU(N)$, it is readily checked that $[\nu]$ is a generator of $\ChainGp(\CQG ) \simeq \Z/N \Z$. In particular, we can apply Proposition \ref{Prop:ResSUN} (ii) for $M = N$ and get $\Falg[\nu] \simeq \Falg[\nu^N]$.

By definition of $\CQG = SU(N)$,  the $N$-fold tensor product $\nu^N$ contains the trivial representation, thus proving Condition \ref{Cond:Entrelac}. A simple dimension counting argument proves that $\nu^N$ must contain another irreducible representation, showing that \ref{Cond:Dim} is satisfied for $\nu^N$. The result follows from Proposition \ref{Prop:DualInvA}.
\end{proof}

\subsection{Example: $SU_q(2)$}
\label{SSec:Example}

In this subsection, we perform detailed computations regarding the case of $\CQG = SU_q(2)$, when $\brep$ is the natural representation of $\CQG $ on $\C^2$. This algebra of fixed points was described explicitly in \cite{ActCMQGKNW} and \cite{CpctQGpMarciniak}. Our results complements the previous ones by providing an identification up to $C^*$-isomorphism.

\medbreak

 In the rest of this subsection, we denote the $n+1$-dimensional irreducible representations of $\CQG$ by $(n)$ for $n \in \N$. Hence, $(0)$ is the trivial representation, $(1)$ is the natural representation... We distinguish between odd and even representations, depending on the parity of $(n)$. This distinction corresponds to the chain group $\ChainGp(SU(2)) = \Z / 2 \Z$ (see \cite{HilbertCSystBaumgLledo}, 5.1.1 p. 794). The tensor product of irreducible representations in $SU_q(2)$ is determined by the Clebsch-Gordan formula (\cite{TwistedSU2Woronowicz} Theorem 5.11):
\begin{equation}
\label{Eqn:CG}
(k) \otimes (\ell) = (|k-\ell|) \oplus (|k-\ell| + 2) \oplus \cdots \oplus (k+ \ell).
\end{equation}
%
Parity is ``compatible'' with \eqref{Eqn:CG} and $(1)$ is odd, thus it is clear that Condition \ref{Cond:Entrelac} cannot be satisfied as such for $\brep  = (1)$. However, we have the following:
\begin{remarque}
\label{Rem:ResSUN}
If $\brep $ is a (nonempty) finite sum of odd representations, then it satisfies all conditions of Proposition \ref{Prop:ResSUN} for $M = 2$ and $\Falg \simeq \Falg[\brep ^2]$. Moreover, the Clebsch-Gordan formula shows that for any nontrivial irreducible representation $(n)$, the tensor product $(n)\otimes (n)$ contains both $(0)$ and $(2 n)$, thus Conditions \ref{Cond:Dim} and \ref{Cond:Entrelac} are satisfied for $\brep^2$ and $N = 0$.
\end{remarque}

We now use $\brep = (1)^2 = (0) \oplus (2)$. As an explicit illustration of the method of Section \ref{Sec:CompKth}, we compute the Bratteli diagram corresponding to $\InvJ^{\brep} = \lim_{\to } \InvJ^{\brep , \ell}$. Since $\brep $ only contains even representations, no odd representation appears in $\CatInd $ (see Notation \ref{Not:TBrep}). Thus there is no odd representation in the Bratteli diagram. Using \eqref{Eqn:CG}, we can compute the iterated tensor powers $\brep ^\ell$:
\begin{align*}
\brep ^0 &= (0)
&
\brep ^2 &= 2.(0) \oplus 3.(2) \oplus (4)\\
\brep ^1 &= (0) \oplus (2)
&
\brep ^3 &= 5.(0) \oplus 9.(2) \oplus 5.(4) \oplus (6)
\end{align*}
From the above computations, we deduce as in \eqref{Eqn:InvJEll} that 
\begin{align*}
\InvJ^{\brep , 0} &= \C
&
\InvJ^{\brep , 2} &= M_2(\C) \oplus M_3(\C) \oplus M_4(\C)\\
\InvJ^{\brep , 1} &= M_2(\C) \oplus \C
&
\InvJ^{\brep , 3} &= M_5(\C) \oplus M_9(\C) \oplus M_5(\C) \oplus \C
\end{align*}
Hence, iterated tensor powers and Bratteli diagram correspond to the same diagram:
$$
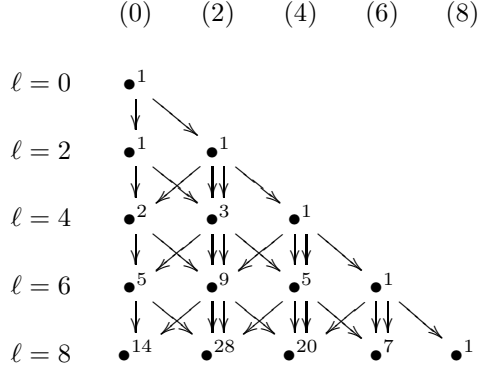
\xymatrix@=4mm{
 & (0) & (2) & (4) & (6) & (8) \\
\ell = 0 & \bullet^1 \ar[dr] \ar[d] & & &&\\
\ell = 2 & \bullet^1 \ar[dr] \ar[d] & \bullet^1 \ar[dl] \ar[dr] \ar@<0.5ex>[d] \ar@<-0.5ex>[d]&&& \\
\ell = 4 & \bullet^2 \ar[dr] \ar[d] & \bullet^3 \ar[dl] \ar[dr] \ar@<0.5ex>[d] \ar@<-0.5ex>[d] & \bullet^1  \ar[dl] \ar[dr] \ar@<0.5ex>[d] \ar@<-0.5ex>[d] & &\\
\ell = 6 & \bullet^5 \ar[dr] \ar[d] & \bullet^9 \ar[dl] \ar[dr] \ar@<0.5ex>[d] \ar@<-0.5ex>[d] & \bullet^5  \ar[dl] \ar[dr] \ar@<0.5ex>[d] \ar@<-0.5ex>[d] &  \bullet^1  \ar[dl] \ar[dr] \ar@<0.5ex>[d] \ar@<-0.5ex>[d] &\\
\ell = 8 & \bullet^{14} & \bullet^{28} & \bullet^{20}  &  \bullet^7 & \bullet^1
}
$$
It is well known that the representation ring of $SU(2)$ can be identified with $\Z[t]$ using $(n) \leftrightsquigarrow U_n(t/2)$ where $U_n$ are the Chebyshev polynomials of the second kind. Explicitly
\begin{align*}
(0) & \leftrightsquigarrow 1
&
(1) & \leftrightsquigarrow t
&
(2) & \leftrightsquigarrow t^2 - 1
&
(3) & \leftrightsquigarrow t^3 - 2 t
&
...&
\end{align*}
One can easily check that the set $\CatInd $ actually contains all even representation. Moreover, the polynomials corresponding to even representations are even. Thus, we can change the variable and use $T := t^2$. In particular $(0) \oplus (2) \leftrightsquigarrow T$.

\smallbreak

From Remark \ref{Rem:Fraction}, we see that $K_0(\InvJ^\brep )$ is a submodule of the fractions $\Z(T)$. It is easy to prove that $((0) \oplus (2))^n$ contains all irreducible representations $(2 k)$ for $0 \leqslant k \leqslant n$. Since the polynomial corresponding to $(2 k)$ has degree $k$ in $T$ and leading coefficient $1$, we get:
\begin{lem}
$K_0(\InvJ^\brep )$ is the set of $\frac{P(T)}{T^n}$ where the polynomial $P(T) \in \Z[T]$ has degree at most $n$.
\end{lem}

We can now conclude:
\begin{prop}
When $\CQG = SU_q(2)$ and $\brep = (1)$, the fixed point algebra $\Falg $ is a Kirchberg algebra in the UCT class $\UCTclass$ whose $K$-theory is
\begin{align*}
K_0(\Falg ) &= \Z
&
K_1(\Falg ) &= 0.
\end{align*}
Moreover, $[1_{\Falg} ]_0 = 1$ and therefore $\Falg $ is $C^*$-isomorphic to the infinite Cuntz algebra $\CAlg_\infty $.
\end{prop}

\begin{proof}
From Proposition \ref{Prop:RedSUN}, we see that all conditions of Theorem \ref{Eqn:KthFalg} and Proposition \ref{Prop:DualInvA} are satisfied. Thus, we get that $K_1(\Falg ) = 0$ and $K_0(\Falg ) = \coker(1 - [E]_*)$. In our situation, it is readily checked that $ [E]_* \left( \frac{P(T)}{T^n} \right) = \frac{P(T)}{T^{n+1}} $. We are then left to evaluate the cokernel of 
$$
\frac{P(T)}{T^n} \mapsto \frac{P(T) (T - 1)}{T^{n+1}}.
$$ 
It is obvious from this expression that $\frac{Q(T)}{T^n}$ is in the image of $1 - [E]_*$ if and only if $Q(1) = 0$. Thus the cokernel is the image of the evaluation map $Q \mapsto Q(1)$, \textsl{i.e.} $\coker(1 - [E]_*) = \Z$. It is also clear that with this identification, $[1_{\Falg} ]_0 = 1$. Applying Theorem \ref{Thm:ClassKirchberg} and known properties of $\CAlg_\infty$, we get the result.
\end{proof}

\subsection{Example: $SU_q(3)$}

In this subsection, we use the notations and results presented in \cite{TensProdWesslen}. The matrix Lie group $\CQG = SU(3)$ is simply connected and thus its representations correspond precisely to those of $\LieAlgSU(3)$. 
We denote by $(p,q)$ the representation with highest weight $p \lambda_1 + q \lambda_2$ where $\lambda_i$ are the fundamental weights. In particular, the trivial representation ($\mathbf{1}$ in Physics notations) is $(0,0)$, the natural representation ($\mathbf{3}$) is $(1,0)$, its contragredient representation ($\mathbf{\overline{3}}$) is $(0,1)$. It follows from \cite{TensProdWesslen} that:
\begin{remarque}
\label{Rem:SU3}
For any irreducible representation $\beta$ of $\CQG = SU(3)$, the tensor product $\beta^{\otimes 3}$ contains the trivial representation $\TrivRep$.

Therefore, if all irreducible representations in $\brep$ have the same class $[\brep ] \neq e$ in $\ChainGp(\CQG )$, then $[\brep ]$ generates $\ChainGp(\CQG ) \simeq \Z / 3 \Z$ and all hypotheses of Proposition \ref{Prop:ResSUN} are satisfied for $M = 3$. Moreover, Conditions \ref{Cond:Dim} and \ref{Cond:Entrelac} are satisfied for $\brep^3$ and $N = 0$.
\end{remarque}

We now use $\brep = (1,0)^3 = (0,0) \oplus 2.(1,1) \oplus (3,0)$. The representation ring of $\CQG $ is $\Z[\Lambda^1, \Lambda^2]$ where $\Lambda^1$ corresponds to $(1,0)$ and $\Lambda^2$ corresponds to $(0,1)$ (see for instance \cite{BrockertomDieck}, Section 5 p.265). By construction, $\brep ^3$ corresponds to $(\Lambda^1)^3$. Since $(0,1)$ is an irreducible component of $(1,0) \otimes (1,0)$, we have:
\begin{lem}
$K_0(\InvJ^\brep )$ is the set of $\frac{P(\Lambda^1, \Lambda^2)}{(\Lambda^1)^{3n}}$ where the polynomial $P(\Lambda^1, \Lambda^2) \in \Z[\Lambda^1, \Lambda^2]$ includes monomials of degree at most $3 n$, where $\Lambda^1$ counts as degree $1$ and $\Lambda^2$ as degree $2$.
\end{lem}

For instance, for $n = 1$, the possible polynomials are:
$$ \frac{c_0 + c_1 \Lambda^1 + c_{2}(\Lambda^1)^2 + c_3 \Lambda^2 + c_4  (\Lambda^1)^3 + c_5 \Lambda^1 \Lambda^2}{(\Lambda^1)^3} $$
where $c_0, \ldots , c_5 \in \Z$. Let us now prove:
\begin{prop}
When $\CQG = SU_q(3)$ and $\brep = (1,0)$, the fixed point algebra $\Falg $ is a Kirchberg algebra in the UCT class $\UCTclass$ whose $K$-theory is
\begin{align*}
K_0(\Falg ) &= \Z^3 \otimes \Z[\Lambda^2]
&
K_1(\Falg ) &= 0.
\end{align*}
\end{prop}

\begin{proof}
Proposition \ref{Prop:RedSUN} shows that all conditions of Theorem \ref{Eqn:KthFalg}, Proposition \ref{Prop:DualInvA} and Theorem \ref{Thm:Main} are satisfied. Thus it suffices to treat the case of $\CQG = SU(3)$.

From Theorem \ref{Thm:KthFalg}, we get that $K_1(\Falg ) = 0$ and $K_0(\Falg ) = \coker(1 - [E]_*)$. In our situation, it is readily checked that $ [E]_* \left( \frac{P(\Lambda^1, \Lambda^2)}{(\Lambda^1)^{3n}} \right) = \frac{P(\Lambda^1, \Lambda^2)}{(\Lambda^1)^{3(n+1)}} $. We are then left to evaluate the cokernel of 
$$
\frac{P(\Lambda^1, \Lambda^2)}{(\Lambda^1)^{3n}} \mapsto \frac{P(\Lambda^1, \Lambda^2)(1 - (\Lambda^1)^3)}{(\Lambda^1)^{3(n+1)}}.
$$ 
It is obvious from this expression that $\frac{Q(\Lambda^1, \Lambda^2)}{(\Lambda^1)^{3(n+1)}}$ is in the image of $1 - [E]_*$ if and only if 
\begin{align*}
Q(1, \Lambda^2) &= 0
&
Q(j, \Lambda^2) &= 0
&
Q({j}^2, \Lambda^2) &= 0.
\end{align*}
Thus the cokernel is given by the evaluation map $\Phi \colon \Z[\Lambda^1, \Lambda^2] \to \C^3 \otimes \Z[\Lambda^2]$:
$$
\Phi(Q) = \Big(Q(1, \Lambda^2), Q(j, \Lambda^2), Q(j^2, \Lambda^2) \Big).
$$
Considering $Q(\Lambda^1, \Lambda^2) = (\Lambda^1)^p (\Lambda^2)^q$ for $p = 0, 1, 2$ and $q \in \N$, we see that the image of $\Phi$ is $\langle v_1, v_2, v_3 \rangle \otimes \Z[\Lambda^2]$ where 
\begin{align*}
v_1 &= (1, 1,1)
&
v_2 &= (1, j, j^2)
&
v_3 &= (1, j^2, j).
\end{align*}
These vectors are free on $\R$ and thus on $\Z$. 
\end{proof}

\section{Final remarks}
\label{Sec:Final}

In this last section, we prove further consequences of our results for semisimple compact Lie groups and their $R^+$-deformations. We first recall the following Proposition 9.3 from \cite{RepTheoryHuang} p.125:
\begin{prop}
\label{Prop:Huang}
Let $\beta$ be an irreducible representation of a semisimple compact Lie group $G$. Each irreducible representation of $G$ occurs in some tensor power $\beta^{\otimes k}$ if and only if $\beta$ is faithful. 

In particular, if $G$ is simple, $\beta$ is faithful if and only if the centre $Z(G)$ is faithfully represented on $\beta$.
\end{prop}

The following first appeared as Definition 2.4 in \cite{PpalActionEllwood} (see also \cite{FreeActCAlgCY}):
\begin{definition}
Given a CQG $\CQG $, an action $\actionG \colon A \to A \otimes C(\CQG )$ on a $C^*$-algebra $A$ is called \emph{free} if $ (A \otimes 1) \actionG(A)$ is dense in $A \otimes C(\CQG )$.
\end{definition}

Going back to our initial motivations concerning free actions, we prove:
\begin{prop}
If $\CQG $ is a semisimple compact Lie group (or a $R^+$-de\-for\-ma\-tion thereof) and $\brep $ is (a $R^+$-deformation of) a faithful representation of $\CQG $ on a $d$-dimensional Hilbert space $\Hh$, then the induced action $\actionG$ on $\CAlg_d$ is {free}.
\end{prop}

\begin{remarque}
This applies in particular to $\CQG = SU(N)$ and its natural representation on $\C^N$.
\end{remarque}

\begin{proof}
First assume that $\CQG $ is a semisimple compact Lie group. We can apply Proposition \ref{Prop:Huang} to $\CQG $ and $\brep $. This way, all irreducible representations of $\CQG $ appear in some tensor power $\brep ^{\otimes k}$ for $k$ large enough. This is also true for any $R^+$-deformation of $\CQG $, since they share the same fusion rules. Thus for any irreducible representation defined by a matrix $(\beta_{ij})$, we can find an orthonormal family of vectors $v_j \in \Hh^{\otimes k}$ s.t.
$$ \actionG(v_j) = \sum_{i} v_i \otimes \beta_{ij} .$$
Realising these vectors in $\CAlg_d$, and picking any $T \in \CAlg_d$ we get:
$$ (T v_i^* \otimes 1) \actionG(v_j) = (T v_i^* \otimes 1) \left( \sum_{k} v_k \otimes \beta_{kj} \right) = T \otimes \beta_{ij} .$$
Relying on Theorem 1.2 of \cite{CpctQGpWoronowicz}, it follows that $(\CAlg_d \otimes 1) \actionG(\CAlg_d \otimes 1)$ is dense in the tensor product $A \otimes C(\CQG )$.
\end{proof}

We conclude with the following:
\begin{remarque}
\label{Rem:CPZ}
The stability result of Theorem \ref{Thm:Main} shows that the fixed point $C^*$-algebra is not a very fine invariant: indeed, it only ``sees'' the fusion rules of $\CQG $. Concretely, in the setting of the natural representation $\nu$ of $SU_q(N)$, we cannot retrieve the ``$q$''.

\medbreak

This situation is the exact opposite of Theorem 2 of \cite{QGpActCPZ}. In an algebraic version of the same setting, this result proves that the (algebraic) fixed point algebra characterises the (algebraic) quantum group. In other words, we can retrieve the ``$q$'' from the algebraic relations.
\end{remarque}

	\bibliographystyle{alpha} 
	\bibliography{biblio}

\paragraph{Acknowledgement}

The author is deeply grateful to D. Bahns for her support and encouragement in general, and for our many discussions about this topic and her careful reading of this article in particular. He also thanks J. Renault as well as R. Meyer and S. Albandik for drawing his attention to \cite{TheseWassermann} and \cite{QfreeAutZacharias}.

\medbreak

Mathematisches Institut -- Universität Göttingen

Bunsenstr. 3-5 D - 37073 Göttingen, Germany 

\textsl{E-mail address:} \texttt{ogabriel@uni-math.gwdg.de}.
\end{document}